\documentclass[reqno,12pt]{amsart}
\usepackage{amsmath,amstext,amsfonts,amssymb}
\renewcommand{\bigcirc}{\bigodot}
\renewcommand{\circ}{\odot}

\def\N{{\mathbb N}}

\def\m{\medskip}

\def\Z{{\mathbb  Z}}

\bigskip

\newtheorem{thm}{\bf Theorem}[section]
\newtheorem{lemma}[thm]{\bf Lemma}
\newtheorem{prop}[thm]{\bf Proposition}

\newtheorem{cor}[thm]{\bf Corollary}
\theoremstyle{definition}
\newtheorem*{defi}{\bf Definition}
\newtheorem{examp}[thm]{Example}

\title[Iterated linear operations]{Iterated compositions of linear operations\\ 
on sets of positive upper density}

\author[N. Hegyv\'ari]{Norbert Hegyv\'ari}
\address{Norbert Hegyv\'{a}ri, ELTE TTK,
E\"otv\"os University, Institute of Mathematics,
H-1117 P\'{a}zm\'{a}ny st. 1/c,
Budapest, Hungary}
\email{hegyvari@elte.hu}

\author[F. Hennecart]{Fran\c cois Hennecart}
\address{Fran\c cois Hennecart,
LaMUSE,
Univ. Jean-Monnet,
23 rue Michelon,
42023 Saint-Etienne cedex 2,
France}
\email{francois.hennecart@univ-st-etienne.fr}

\thanks{Research of the first and second named authors
are supported by ``Balaton Program Project" and OTKA grants TO
43623,49693,38396. }

\author[A. Plagne]{Alain Plagne}
\address{Alain Plagne,
Centre de Math\'ematiques Laurent Schwartz,
UMR 7640 du CNRS,
\'Ecole polytechnique,
91128 Palaiseau Cedex,
France}
\email{plagne@math.polytechnique.fr}

\date{\today}

\begin{document}

\begin{abstract}
Starting from a result of Stewart, Tijdeman and Ruzsa 
on iterated difference sequences, we introduce the notion 
of iterated compositions of linear operations. We prove 
a general result on the stability of such compositions (with bounded 
coefficients) on sets of integers having a positive upper density.
\end{abstract}

\maketitle

\section{\bf Introduction}

Let $G$ be an additive Abelian group considered as a $\Z$-module. 
A {\em linear operation} $\Gamma$ is a mapping 
$$
X\mapsto aX+bX:= \{ax+bx'\mid x,x'\in X\}\hspace{1cm}  (X \subset G)
$$
from the set of all subsets of $G$ on itself,
where $a,b\in \Z$ are fixed integers.
We also introduce the concept of {\em iterated linear operation} in the following way:
a linear operation $\Gamma$ being given, we put  $\Gamma_{1}=\Gamma$, and
for $k\ge2$, $\Gamma_k(X)=\Gamma(\Gamma_{k-1}(X))$ for any $X\subset G$.
An important example of linear operation is given by  the difference operation defined by
$\Gamma(X)=X-X$, ($X\subset G$). 

In the case where $G$ is the set of integers, Stewart and Tijdeman in
\cite{r4} investigated the so-called iterated {\em positive} difference 
operation: for an infinite set  $A$ of positive integers,
let $D^+(A)$ be the positive difference set
defined by 
$$
D^+(A)=\{a-a'\mid a\geq a',\, a,a'\in A\}.
$$ 
The sequence of iterated positive difference sets 
$\{D^+_k(A);\;k\ge0\}$ of $A$ is defined by 
$D^+_{0}(A)=A$ and $D^+_k(A)=D^+(D^+_{k-1}(A))$ for $k\ge1$.
Stewart and Tijdeman observed
that if a sequence $A$ has positive {\em upper density} i.e. if
$$
\overline{d}(A):= \limsup_{n \to \infty}\frac{| A\cap\{1,2,\dots,n\} |}{n}>0,
$$
then the sequence $\{D^+_k(A);\; k\ge0\}$ is stable, that is there exists 
an integer $k_0$ such that, $D^+_{k+1}(A)=D^+_k(A)$ for every $k\geq k_0$.
The time of stability of $A$ is defined by $T(A)=\min \{k\mid
D^+_{k+1}(A)=D^+_k(A)\}$.
For instance, if $\overline d (A)>1/2$, it is readily seen that
$D^+(A)$ is the whole set of nonnegative integers, hence
$T(A)\le1$. In \cite{r4} Stewart and Tijdeman gave an upper bound for
$T(A)$ if the upper density of $A$ is positive. They proved that if
$0<\overline d (A)\leq 1/2$ then $T(A)\leq
2\log_2 (\overline{d}(A)^{-1})$, where $\log_{2}$ denotes the logarithmic function in
base $2$. This result was improved by Ruzsa in \cite{r8} where it is shown
that under the same assumption on $\overline{d}(A)$, we have
$T(A)\leq 2+\log_2(\overline {d}(A)^{-1}-1)$.

Instead of a restricted difference operation, we may also investigate the related question
of the stability of the sequence $\{D_{k}(A);\;k\ge0\}$, with $D_{0}(A)=A$, 
$D_{1}(A)=A-A$ and $D_{k}(A)=D(D_{k-1}(A))$, for $k\ge1$.  
The advantage of this question is that it
can be handled  in more general groups, as shown in
\cite{r5} and \cite{r9}. As a direct consequence of Stewart-Tijdeman's or Ruzsa's results,
we infer the stability of $\{D_{k}(A);\;k\ge0\}$ whenever
$A$ has a positive upper density in the set of positive integers.

For $n\in \N$ and $X$ a subset of some (additively written) group $G$, 
we shall use the following (slightly non standard) notation
$$
nX=\{nx \mid  x\in X\},\quad Xn=\underbrace
{X+X+\cdots+X}_{n\text{ times}}.
$$
It is easy to see that for $n,m\in \N$ we have $(mX)n=m(Xn)$ so we
briefly write  $mXn$. Furthermore for $n,m,k\in \N$
$$
nXk+nXm=nX(k+m),\quad nXk+mXk\supseteq (n+m)Xk.
$$

For a real number $z$, we shall also use the notation
$$
||z||=\min_{m \in \Z} |m-z|\quad 
$$
and $\lfloor z \rfloor$ (resp. $\lceil z \rceil$) for the
integral part of $z$ by default (resp. by excess). Finally let 
$\{z\}=z-\lfloor z \rfloor$ be the fractional part of $z$.

In this paper, we restrict our attention to iterated linear
operations in the set of integers without any other restriction.

\section{\bf A preliminary discussion and plan of the paper}

Let $a$ and $b$ be given integers and $\Gamma$ be the linear operation
defined on subsets $X \subset \Z$ by 
$$
\Gamma(X)=aX+bX=\{ax+bx'\mid x,x'\in X\}.
$$
As before, we set $\Gamma_{0}(X)=X$ and
$\Gamma_k(X)= \Gamma(\Gamma_{k-1}(X))$, for $k\ge1$. 
The central question in this paper is: What can be said 
on the stability of the sequence $\{\Gamma_k(X);\;k\ge1\}$ 
if we assume further that $X$ has a positive upper density ?

The case $(a,b)=(1,-1)$ leads to the usual difference set 
and the stability ensues from Stewart-Tijdeman's and Ruzsa's results 
on the iterated positive difference operation. 

If $a b>0$, then  the absolute value of the minimal element of $\Gamma_k(X)$ 
tends to infinity as $k$ tends to infinity and consequently the sequence 
$\{\Gamma_k(X);\;k\ge1\}$ cannot be stable. 
Therefore, without loss of generality, we may
assume that $\Gamma(X)=aX-bX$ with $ab>0$.
Since for every integer
$\alpha$  we have $\Gamma(\alpha X)=\alpha\Gamma(X)$, we get
$\Gamma_k(-X)=-\Gamma_k(X)$,  for any $k\in\mathbb{N}$, which implies 
that $\Gamma_k(-X)$ and  $-\Gamma_k(X)$ have the same structure. It is
thus enough to consider the case $a>b>0$.

In the case $a=b+1$, it is not hard to show that  for any
arithmetic progression $X$, the sequence $\{\Gamma_k(X);\;k\ge1\}$ is stable.

We now consider the case when $a> b+1$. In this case, we
show that there exists an arithmetic progression $X$ for which
the sequence $\{\Gamma_k(X);\;k\ge1\}$ is not stable. 
For this, we distinguish two cases:
\medskip

Case 1: $\gcd(a,b)=d>1.$\

Let $a'=a/d$ and $b'=b/d$. Then $\Gamma_1(X)=\Gamma(X)=d(a'X-b'X)=dX_1$, 
with $X_{1}\subset\mathbb{Z}$, and by induction we have
that  $\Gamma_k(X)=d^k X_k$ for every $k\in \N$, for some
set $X_k\subset \Z$. In this case, we can choose $X$ to be $\N$. Now
if the sequence $\{\Gamma_k(X);\;k\ge1\}$ were stable then
there would be an interval $(-\alpha, \alpha)$ for which
$(-\alpha, \alpha)\cap \Gamma_k(X)\neq \varnothing$ for every
large $k$. This is a contradiction to $\Gamma_k(X)=d^k X_k$.
\smallskip

Case 2: $\gcd(a,b)=1.$\

Let $X=\{ab m+1: m\in \N\}$. We claim that
$$
\Gamma_k(X)=\{abm+(a-b)^k\mid m\in \Z\},\qquad k\in\mathbb{N}.
$$
Indeed, we first note that  
$$
\Gamma_1(X)=\{a(abm+1)-b(abn+1) \mid m,n\in\N\}
=\{ab(am-bn)+a-b \mid m,n\in \N\},
$$ 
and since $\gcd(a,b)=1$, we obtain $\Gamma_1(X)=\{abm+a-b \mid m\in \Z\}$. 
We  get by induction $\Gamma_k(X)=\{abm+(a-b)^k \mid m\in \Z\}$, $k\in\N$.

If $\Gamma_{k+1}(X)=\Gamma_{k}(X)$ for some $k$, we infer 
$(a-b)^{k+1}\equiv (a-b)^k \pmod{ab}.$ But since $\gcd(a-b,ab)=1$, 
this implies that $a-b\equiv 1 \pmod{ab}$ and since $a,b+1\leq
ab$ we obtain $a=b+1,$ a contradiction.  Thus the sequence
$\{\Gamma_k(X);\; k\ge1\}$  cannot be stable.

\renewcommand{\theenumi}{\roman{enumi}}
\medskip

In the above construction -- when $\gcd(a,b)=1$ and
$a\neq b+1$ -- the sequence $\{\Gamma_k(X);\;k\ge1\}$ is not stable, but
it has a regularity property: $\{\Gamma_k(X);\;k\ge1\}$ is eventually
periodically stable in the sense that there exists a positive integer 
$p$ such that $\Gamma_{k+p}(X)=\Gamma_{k}(X)$ for any integer $k$.

According to this observation, we will extend
the notion of stability in the more general context of composition 
of linear operations described in Section \ref{S2}. 
We will investigate the stability of sequences defined by iterating 
a priori distinct linear operations $X\mapsto a_{k}X-b_{k}X$ ($k\ge1$), 
on a set $X$ of integers.

In Section \ref{S3}, several useful results in our context, on density and gaps
will be presented, while in Section \ref{S4} an inverse result (Proposition 
\ref{proposition5.3}) for linear operations on a set of residues classes 
modulo some integer will be stated and proved.

Having all this material at hand, we will be able in Section \ref{S5} 
to state our main result (Theorem \ref{theorem6.1}). This result in particular 
implies that if one iterates linear operations with bounded coefficients on a 
set of integers with positive upper density, then the resulting set of integers 
will be fully periodic from some time on. Moreover the sequence of iterates 
will be stable. In the special case of iterating a unique linear operation, 
then the sequence of iterates will be not only stable but itself periodic 
(see Remark (1) in the final section).

\section{\bf Composition of linear operations and stability}
\label{S2}

Instead of iterating  a unique linear operation $\Gamma$ as discussed up to now, 
we consider a composition of different linear operations in the following way:

\m
For $(a_{1},b_{1})$ and $(a_{2},b_{2})$ two couples of positive integers,
let $\Gamma_{a_{j},b_{j}}$, $(j=1,2)$, be defined by
$\Gamma_{a_j,b_j}(X)=a_j X - b_j X$, ($X\subset \mathbb{Z}$).
The  composition
of $\Gamma_{a_1,b_1}$ and $\Gamma_{a_2,b_2}$,  denoted by
$\Gamma_{a_1,b_1}\circ \Gamma_{a_2,b_2}$ is the linear operation defined on
each set $X \subset \Z$ by
$$
\Gamma_{a_1,b_1}\circ \Gamma_{a_2,b_2}(X)=
\Gamma_{a_2,b_2}(\Gamma_{a_1,b_1}(X))=
a_{1}a_{2}X+b_{1}b_{2}X-a_{1}b_{2}X-a_{2}b_{1}X.
$$
More generally, let $(a_1,b_1),(a_2,b_2),
\dots (a_{s},b_{s})$ be a finite sequence of couples of positive integers and define
the composition of the $\Gamma_{a_{j},b_{j}}$, ($j=1,\dots,s$) in a natural way by
$$
\bigcirc_{j=1}^s \Gamma_{a_j,b_j}(X)=\Gamma_{a_1,b_1}\circ
\Gamma_{a_2,b_2}\circ \cdots
\circ\Gamma_{a_s,b_s}(X).
$$
For $s=0$, this convoluted set is defined to be $X$.

We now give an important definition for our purpose.

\begin{defi}
{\em
Let $t$ be a positive integer and $(a_{j},b_{j})_{j\in\mathbb{N}}$, 
be a sequence of couples of positive integers.
We say that a subset $X\subset \N$ is $t$-stable
with respect to the sequence of linear operations 
$(\Gamma_{a_j,b_j})_{j\in\N }$ if  the set
$\{X\}\cup\{\bigcirc_{j=1}^{k} \Gamma_{a_j,b_j}(X) \mid k\in\N\}$
has a cardinality less than or equal to $t$.
}
\end{defi}

We expect that a $t$-stable sequence has a ``big'' upper
density. An integer $s$ being given, we write $[1,s]$ 
for the set $\{1,2,\dots,s\}$. The notation $I\sqcup J=[1,s]$ 
means that $I\cup J=[1,s]$ and $I\cap J=\varnothing$. 

Let us first prove the following result.

\begin{thm}
\label{theorem3.1}
Let $t$ be any positive integer and 
$(a_1,b_1),(a_2,b_2), \dots, (a_{t},b_{t})$
be $t$ couples of positive integers. 
Then there exists a set $A\subset \N$ with asymptotic density
$$
d(A):= \lim_{n \to \infty}\frac{| A\cap\{1,2,\dots,n\}|}{n} = 
\prod_{i=1}^t\frac1{a_{i}+b_{i}}
$$
such that the finite set 
$\{\bigcirc_{j=1}^{s} \Gamma_{a_j,b_j}(A) \mid 0\le k\le t\}$
has cardinality $t+1$.
\end{thm}

As an immediate consequence, we have the following corollary. 

\begin{cor}
\label{corollary3.2}
Let $L\ge1$ be an integer and $(a_j,b_j)_{j\in\mathbb{N}}$, 
be a sequence of positive integers such that $|a_j|,|b_j|\le L$ for any $j\ge1$.
Then for any positive integer $t$, there exists a set $A\subset \mathbb{N}$ 
with  asymptotic density $d(A)\ge (2L)^{-t}$ such that $A$ is not 
$t$-stable with respect to $(\Gamma_{a_j,b_j})_{j\in\N }$.
\end{cor}

This result shows that if one demands to a set $A$ to be $t$-stable 
then $t$ has to be large enough (with respect to the density of $A$).
It is to be seen as a kind of limit (or a counterpart) to our 
main forthcoming result, namely Theorem \ref{theorem6.1}.

Before giving the very proof of Theorem \ref{theorem3.1}, we start 
with two lemmata. First, the following lemma 
can be obtained by a straightforward induction.

\begin{lemma}
\label{lemma3.3}
We have
$$
\bigcirc_{j=1}^s \Gamma_{a_j,b_j}(X)=
\sum _{I\sqcup J=[1,s]}(-1)^{|J|}\big (\prod_{i\in I}a_i\cdot \prod_{j\in J}b_j\big )X.
$$
\end{lemma}

Note that this lemma implies that composition of linear operations 
is commutative, and in particular
$
\Gamma_{a_1,b_1} \circ \Gamma_{a_2,b_2}(X)
=\Gamma_{a_2,b_2} \circ \Gamma_{a_1,b_1}(X).
$

We shall also need the following immediate metrical lemma.

\begin{lemma}
\label{lemma3.4}
Let $X$ and $Y$ be two sets of positive integers and $\alpha$ be a real number.
If the sequences $(\{ \alpha x \})_{x\in X}$ and  $(\{ \alpha y \})_{ y\in Y}$
are dense in $(1-\beta,1)\cup(0,\beta)$ and $(1-\gamma,1)\cup(0,\gamma)$ respectively 
then the sequence  $(\{ \alpha z \})_{z\in X+Y}$ is dense in
$(1-\mu,1)\cup(0,\mu)$ where $\mu=\min(\beta+\gamma,\frac12)$.

Moreover for any integer $a$, the sequence $(\{ \alpha ax \})_{x\in X}$
is dense in $(1-\lambda,1)\cup(0,\lambda)$ where 
$\lambda=\min(|a|\beta,\frac12)$.
\end{lemma}

We are now prepared for the proof of Theorem \ref{theorem3.1}.

\begin{proof}[Proof of Theorem \ref{theorem3.1}]
In this proof, we shall write $\delta:=1/\prod_{i=1}^t (a_{i}+b_{i})$. 
We define  $A$ as follows: let $\alpha$ be a positive irrational real number and let
$$
A=\left\{a\in \N; \| \alpha a\| <\frac{\delta}2\right\}.
$$
Clearly $d(A) = \delta$. Let $s\le t$ and $x\in\bigcirc_{j=1}^s \Gamma_{a_j,b_j}(A)$. 
By Lemma \ref{lemma3.3} we can write
$$
x= \sum _{I\sqcup J=[1,s]}(-1)^{|J|}
\left (\prod_{i\in I} a_i \cdot \prod_{j\in J}b_j\right) u_{J}
$$
with $u_{J}\in A$, $J\subset[1,s]$. We first observe that
\begin{align*}
\|\alpha x\|&= \left\|\sum _{I\sqcup J=[1,s]}
(-1)^{|J|}\left (\prod_{i\in I}a_i\cdot \prod_{j\in J}b_j\right)
\alpha u_{J}\right\|\\
&\leq
\sum _{I\sqcup J=[1,s]}
\left (\prod_{i\in I}a_i\cdot \prod_{j\in J}b_j\right)
\|\alpha u_{J}\|\\
&<
\sum _{I\sqcup J=[1,s]}
\left (\prod_{i\in I}a_i\cdot \prod_{j\in J}b_j\right) \frac{\delta}2\\
&=\frac{\delta}2\prod_{i=1}^s(a_{i}+b_{i})=
\frac12\prod_{i=s+1}^t\frac{1}{a_{i}+b_{i}}.
\end{align*}
More precisely, since $(\{\alpha a\} )_{a\in A}$ is a dense subset 
of $(1-\delta/2,1)\cup(0,\delta/2)$, 
by Lemma \ref{lemma3.4} and by arguing inductively  we infer that,
for any $s\le t$, the set
$$
\{\|\alpha x\|: x\in \bigcirc_{j=1}^s \Gamma_{a_j,b_j}(A)\}
$$
is a dense subset of  $(0,\frac12\prod_{i=s+1}^t(a_{i}+b_{i})^{-1})$.
Thus clearly all the sets $\bigcirc_{j=1}^s \Gamma_{a_j,b_j}(A)$, $0\le s\le t$, are
mutually distinct.
It follows that $A$ is not $t$-stable.
\end{proof}

An efficient tool that can be used for
yielding the stability of iterated difference sets is
Kneser's theorem (cf. Lemma \ref{lemma4.4}) which describes for $h$ large enough
the structure of any $h$-fold sumset of a sequence of integers having 
a positive lower density. Indeed, if $X$ is assumed to have a positive upper density,
then the first difference set $X-X$ of $X$ is in fact well distributed,
in the sense that it has a positive lower density, namely 
$$
\underline{d}(X-X):= \liminf_{n \to \infty}\frac{| ( X-X)\cap\{1,2,\dots,n\}|}{n} >0
$$
since its {\em gaps} are bounded. Recall that the gaps of an increasing 
sequence $(u_n)$ is the sequence $(u_{n+1}-u_n)$. 
A short proof of this fact is as follows: By a finite recursive construction, 
we first find a maximal set of integers $T=\{ t_1,t_2,\dots,t_s \}$ such that 
the translated sets $X+t_i$ of $X$ are pairwise disjoint ($s$ is finite and 
more precisely must be bounded from above by $1/\overline{d}(X)$). 
Then any integer $z$ is such that $X+z$ intersects at least one of the $X+t_i$'s 
and therefore can be written as $z=(x-x')+t_i$ for some $1 \leq i \leq s$ and 
$x,x'\in X$. We consequently infer that $T+(X-X)=\Z$. In particular, a gap 
in $X-X$ cannot be larger than  $\max_{1 \leq i \leq s} t_i$.

It is no more the case when $(a,b)\neq(1,1)$ as shown by 
the following example where we give
a set $A$ such that $\overline{d}(A)>0$
and $aA-bA$ has arbitrary large gaps.

\begin{examp}
\label{example3.5} 
Let $(a,b)\neq (1,1)$  and
$$
A=\bigcup_{i=1}^\infty(x_i,x_i(1+\delta))\cap
\N,
$$
where $\{x_{i};\;i\ge1\}$ is any fast increasing sequence of positive real numbers
(for instance $x_{i}=i^i$).
If $\delta<a/b-1$, then $\Gamma_{a,b}(A)$ has arbitrary large gaps 
while $\overline{d}(A) \geq \delta/(1+\delta)$.
\end{examp}

Nevertheless, we shall see in Lemma \ref{lemma4.2} that under an additional
hypothesis implying $a,b$ and $\overline{d}(X)>0$, the set $aX-bX$
has bounded gaps and thus has a positive lower density.

\m
A nice result of Bergelson and Ruzsa \cite{r7} brought to our knowledge
in a personal communication generalizes a
theorem of Bogolyubov; these authors proved that if $(r,s,t)$ is 
a triple of integers with $r+s+t=0,$ and $\overline{d}(X)>0$ 
then the set $rX+sX+tX$
contains a Bohr set, that is a set of integers of the type
$\{n\in\mathbb{N}\,:\, \|\alpha_{i}n\|\le\varepsilon_{i},\;
i=1,2,\dots,r\}$, where $\alpha_{i}$, $1\le i\le r$, are
given real numbers, and $\varepsilon_{i}$, $i=1,\dots,r$, are
positive real numbers.

If $X$ is the set mentioned in Example \ref{example3.5}, 
we see that for $(r,s,t)=(a,-b,0)$, we have
$r+s+t\neq 0$ and the set $rX+sX+tX$ will not contain a Bohr set since
it has arbitrary large gaps (while a Bohr set has bounded gaps).

\section{\bf Additive tools}
\label{S3}

We will need the following consequence of a result by Freiman known as
Freiman's $3k-3$ Theorem. It asserts that for a given finite set $X$ of $k$ mutually
coprime nonnegative integers
containing $0$ with largest element $m$, one has $|X+X|\ge\min(3k-3,k+m)$.

\begin{lemma}
\label{lemma4.1}
Suppose that $X\subset \N$, $0\in X$ and $\gcd(X)=1$. Then

{\rm(i)} if $\overline{d}(X)\le 1/2$, then
$\overline{d}(X+X)\geq 3\overline{d}(X)/2,$

{\rm(ii)} if $\overline{d}(X)>1/2$, then
$\overline{d}(X+X)\geq (1+\overline{d}(X))/2$.
\end{lemma}

This statement is known in the folklore (see for example \cite{r6}). For the sake
of completeness we give a proof of it now.

\begin{proof}
Let $\varepsilon>0$ and let $(n_k)_{k\ge1}$ be a sequence of positive
integers for which $|X_k|/n_{k}>\overline{d}(X)-\varepsilon$ where
$X_k:=X\cap [1,n_k]$ and $n_{k}\in X_{k}$. In both cases,
the result will follow from Freiman's $3k-3$ Theorem:

(i) First assume that $\overline{d}(X) <1/2$.
Since we have $\gcd(X_k)=1$ and clearly $n_{k}=\max(X_k)\geq 2|X_{k}|-3$ if
$k$ is large enough, Freiman's $3k-3$ Theorem yields $|X_k+X_k|\geq 3|X_k|-3.$
Since $X_k+X_k$ lies in $[0,2n_k]$, we have  
$$
\frac{|X_k+X_k|}{2n_k}\geq
\frac{3|X_k|}{2n_k}-\frac{3}{2n_k}>\frac{3}{2}
\overline{d}(X)-\varepsilon-\frac{3}{2n_k}
$$ 
which implies the statement.

(ii) Here we suppose $\overline{d}(X) >1/2$.
Let $\varepsilon$ be sufficiently small. 
We have $\max(X_k)\leq 2|X_{k}|-4$ for any $k$ large enough.
By Freiman's $3k-3$ Theorem again, we get
$|X_k+X_k|\geq |X_k|+n_{k}$, thus
$$
\frac{|X_k+X_k|}{2n_k}\geq
\frac{|X_k|}{2n_k}+\frac12\ge\frac{\overline{d}(X)+1-\varepsilon}2.
$$

To complete this proof, it remains to treat the case $\overline{d}(X)=1/2$.
As above, we get for $k$ sufficiently large 
$|X_{k}+X_{k}|\ge\min(|X_{k}|+n_{k},3|X_{k}|-3)$,
thus 
$$
\frac{|X_{k}+X_{k}|}{2n_{k}}\ge
\min\left( \frac{3-2\varepsilon}4,\frac34-\varepsilon-\frac{3}{2n_k}\right),
$$
and the result follows.
\end{proof}

The following lemma generalizes a previous result obtained 
by Stewart and Tijdeman in \cite{r4}.

\begin{lemma}
\label{lemma4.2}
Let $X\subset \N$ and 
$a,b\in \N,$ such that $a\ge b\ge1$ and $\overline{d}(X)>a/(a+1)$. 
Then the gaps in both  sets $\Gamma_{a,b}(X)=aX-bX$  and 
$\Gamma_{b,a}(X)=bX-aX$ are
bounded from above by $a$.
\end{lemma}

\begin{proof}
We first focus our attention to the set $\Gamma_{a,b}(X)=aX-bX$.

Let $n$ be a positive integer and put $t= n/b$. We define
$Y:=X\cap ( n/a + 1,kb]$ where the integer $k$ is large 
enough in order to have $|Y|>(1-\delta)kb$ where $\delta$ 
is chosen such that $1/(a+1) \ge \delta > 1-\overline{d}(X)$. Let
$$
Z:=\bigcup_{y\in Y}\left[(y-1)\frac ab -t,y\frac ab -t\right).
$$
Observe that $Y$ and $Z$ are subsets of $[1,ka]$ and that
$$
|Z|\ge \left\lfloor\frac ab \right\rfloor  |Y| 
> \left\lfloor\frac ab \right\rfloor(1-\delta)kb.
$$
If $Y\cap Z=\varnothing$, then we would have
$$
\left\lfloor\frac ab\right\rfloor(1-\delta)kb+(1-\delta)kb<ka,
$$
giving $(1-\delta)(\lfloor a/b \rfloor+1)< a/b$, a contradiction
to our assumption $\delta\le 1/(a+1)$.
Thus $Y\cap Z\ne\varnothing$. Hence there exist $y',y''\in Y$
such that 
$$
y'\in\left[(y''-1)\frac ab -t,y''\frac ab -t\right).
$$
We clearly thus have
$$
1\le ay''-by'-n\le a.
$$
This implies that for any positive integer 
$n\in aX-bX$ we can find an element $n'\in aX-bX$ such that
$1\le n'-n\le a$. 

The result for the set $bX-aX$ can be obtained by arguing similarly with 
$Y:=X\cap[1,kb-n/a-1]$ and 
$Z:=\bigcup_{y\in Y}\left(ya/b +t,(y+1) a/b +t\right]$.

This completes the proof of the lemma.
\end{proof}

\begin{lemma}
\label{lemma4.3} 
Let $t$ be any positive integer and $(a_1,b_1),(a_2,b_2), \dots, (a_{t},b_{t})$
be $t$ couples of positive integers. 
Assume that, for every $1 \le i \le t$, we have $1\leq a_i, b_i\leq L$,
for some integer $L \ge 2$.
Let $A$ be any set of nonnegative integers and $m \in \N$.

If $t \ge 2\log_{2} (m) + 4L + 2$ then there exist two positive integers
$\alpha\le L^t$ and $\beta\le L^t$ such that
$$
\bigcirc_{j=1}^t \Gamma_{a_j,b_j}(A) = \alpha Am -\beta Am+B,%\eqno (17)
$$
for some set of integers $B$.
\end{lemma}

\begin{proof}
An arbitrary ``coefficient" $\prod_{i\in I}a_i\cdot \prod_{j\in J}b_j$ 
appearing in the decomposition of $\bigcirc_{j=1}^t \Gamma_{a_j,b_j}(A)$ 
given by Lemma \ref{lemma3.3} namely
$$
\bigcirc_{j=1}^t \Gamma_{a_j,b_j}(X)=
\sum _{I\sqcup J=[1,t]}(-1)^{|J|}\big (\prod_{i\in I}a_i\cdot 
\prod_{j\in J}b_j\big )X
$$
can be written in the form $2^{\gamma_{2}}3^{\gamma_{3}}4^{\gamma_{4}}\dots
L^{\gamma_{L}}$ where the $\gamma_{i}$ are nonnegative  integers 
such that $\gamma_{2}+\gamma_{3}+\gamma_{4}+\cdots+\gamma_{L}\le t$.
Since  the number of $(L-1)$-uples 
$(\gamma_{2},\dots,\gamma_{L})$ satisfying the previous
conditions is less than or equal to $\binom{t+L-1}{L-1}$, 
it is bounded by 
$$
\binom{t+L}{L}\le  \frac{(t+L)^L}{L!} \le \left( \frac{et}{L} +e \right)^L 
\le \left( \frac{4t}{L} \right)^L
$$ 
by easy considerations and using $t\ge4L$ in the last inequality.
Hence there are at most $(4t/L)^L$ values which can be taken 
by a  ``coefficient" $\prod_{i\in I}a_i\cdot \prod_{j\in J}b_j$.

Thus in the decomposition of $\bigcirc_{j=1}^t \Gamma_{a_j,b_j}(A)$ 
given by Lemma \ref{lemma3.3}
(as the sum of the $2^{t-1}$ terms with a positive coefficient 
and $2^{t-1}$ terms with a negative one), there is some positive 
``coefficient'' denoted by $\alpha$,  and some negative
``coefficient'' denoted by $-\beta$  such that  
$1\le \alpha ,\beta \leq L^t$ and  which can be obtained in at least
$$
\left\lceil \frac{2^{t-1}}{(4t/L)^L} \right\rceil  %\eqno (18)
$$
ways. 

Observe now that if $u$ and $x$ are two positive real numbers such that
$u\ge 2\log_{2}(x)+4$ and $x\ge1$ then $2^u/u\ge 4x$. By applying this
with $u=t/L$ and $x=(2m)^{1/L}$, we get that 
$2^t / (4t/L)^L \ge 2m$ as far as $t\ge 2\log_{2}(2m)+4L$.
Hence the result.
\end{proof}

We end this section by stating without a proof 
a fitted version of Kneser's theorem 
for addition of increasing sequences of integers (see \cite{HR}). 
   
\begin{lemma}[Kneser]
\label{lemma4.4}
Let $X\subset\N$ and  $k$ be a positive integer. Assume that
$\underline{d}(X)>0$. Then either
$$
\underline{d}(Xk)\geq k\underline{d}(X),
$$
or there is a positive integer $g$ and a set $X' \subset \N$ 
satisfying $X'+g \subset X'$ such that $X\subset X',$ all sufficiently 
large elements of $X'k$ are in $Xk$, and
$$
\underline{d}(Xk)\geq k\underline{d}(X')-\frac{(k-1)}{g}.
$$
\end{lemma}

\section{\bf An inverse result for linear operations on a set of residues}
\label{S4}

For a given subset $U$ of an abelian group $G$ we denote
by $P(U)$ the maximal subgroup $H$ of $G$ such that $U+H=U$.
We call $P(U)$ the {\em period} of $U$.
The set $U$ is said to be periodic if $P(U)$ is not the trivial group $\{0\}$.

For a given positive integer $g$, 
a set $A$ of integers is said to be periodic or {\it semi-periodic} modulo $g$ if 
$A+g\subset A$. It is said {\it fully} periodic modulo $g$ if $A+g=A$, 
that is $A$ is a reunion
of complete arithmetic progressions modulo $g$ (notice, in particular,
that a fully periodic set of integers must be unbounded both from
below and from above). If $A$ is fully periodic modulo 
$g$, then $A+A'$ is also fully periodic modulo $g$ 
for any set $A'$ of integers.

\begin{lemma}
\label{lemma5.1}
Let $A$ and $A'$ be set of integers which are  semi-periodic 
modulo $g$ and $g'$ respectively. Then $A-A'$ is fully periodic modulo $\gcd(g,g')$.
\end{lemma}

\begin{proof}
Denote by $d$ the greater common divisor of $g$ and $g'$.
Then there exist nonnegative integers $u$ and $v$ such that $ug-vg'=d$
hence $A-A'+d\subset A-A'$. There exist also nonnegative integers
$u'$ and $v'$ such that $u'g-v'g'=-d$, hence $A-A'-d\subset A-A'$.
From this double inclusion, we conclude that $A-A'+d=A-A'$, as asserted.
\end{proof}

One easily sees that if $U$ is a subset of some abelian group $G$
such that $|U+U|=|U|$ then $U$ is a coset modulo some subgroup $H$ of $G$.
For $a$ and $b$ coprime, we will show a structure
result for the subsets $U$ of $\mathbb{Z}/g\mathbb{Z}$ such that $|aU+bU|=|U|$.
We first prove the following lemma.

\begin{lemma}
\label{lemma5.2}
Let $g$ be a positive integer and $X$ be a subset  of
$G=\mathbb{Z}/g\mathbb{Z}$  containing $0$.
Let $a$ and $b$ be two positive integers such that $\gcd(a,b)=1$. We assume that
$X$ is not periodic 
and that  $aX+bX=aX$. Then 
$$
X\subset\frac{g}{\gcd(g,b)}G.
$$
\end{lemma}

\begin{proof}
Let $p$ be any prime factor of $\gcd(g,a)$ and write $g=p^{\alpha}m$
with $p\nmid m$. In view of $aX \subset pG$ and $aX+bX=aX$, we have 
$aX+bX\subset pG$. 
Since $aX$ is composed of multiples of $p$, we then must have 
$bX\subset pG$. Thus $X\subset pG$ since $p\nmid b$. 
By a straightforward induction, we get $X\subset p^{\alpha}G$. 
Taking into account each prime factor of $g$, we obtain $X\subset a'G$
where
$$
a'=\prod_{\substack{p\mid\gcd(a,g)\\
p^{\alpha}\|g}}p^{\alpha}.
$$

Now, the set $X$ can be lifted in $\mathbb{Z}$ into a set $a'Z$ of
multiples of $a'$ for which we have
$aZ+bZ=aZ$ modulo $g/a'$. Since $a$ and $g/a'$
are coprime, we can find an integer $a''$ 
such that $aa''\equiv1$ modulo  $g/a'$. 
We deduce therefore $Z+a''bZ=Z$ modulo $g/a'$, yielding $X+a''bX=X$.
Since $X$ is not periodic, it follows that $a''bX=\{0\}$ and, 
in view of $\gcd(g,a'')=1$, $bX=\{0\}$. This implies $X\subset (g/\gcd(g,b))G$.
\end{proof}

\renewcommand{\theenumi}{\roman{enumi}}

\begin{prop}
\label{proposition5.3}
Let $g$ be a positive integer and $U$ be a subset of $G=\mathbb{Z}/g\mathbb{Z}$.
Let $a$ and $b$ be two positive integers such that $\gcd(a,b)=1$.
Then
\begin{enumerate}
\item For any subgroup $H$ of $G$, we have $aH+bH=H$,

\item $|aU+bU|\ge|U|$,

\item Assume that $0\in U$, that $U$ is not periodic and that $U$ is not included in a
proper (i.e. $\neq G$) subgroup of $G$.
Then the equality $|aU+bU|=|U|$ occurs if and only if
$g=\gcd(g,a)\times\gcd(g,b)$ (or equivalently $g\mid ab$) 
and if there exist two sets $V\subset \gcd(g,b)G$ and
$X\subset \gcd(g,a)G$
such that $U=V+X$ and $|U|=|V|\times|X|$,

\item Assume that $0\in U$ and that $U$ is not included in a
proper subgroup of $G$.
Then the equality $|aU+bU|=|U|$ occurs if and only if
there exist two integers $a_{1}$, $b_{1}$ and two subsets $V$, $X$ of $G$
such that $a_{1}\mid\gcd(g,a)$, $b_{1}\mid\gcd(g,b)$, $V\subset a_{1}G$,
$X\subset b_{1}G$ and $U=V+X+a_{1}b_{1}G$ with
$|U|=|V|\times|X|\times|a_{1}b_{1}G|$,

\item If $|aU+bU|=|U|$ then the period of $aU+bU$ coincides with that of $U$.
\end{enumerate}
\end{prop}

\begin{proof}
(i) Since any subgroup of a cyclic group is also cyclic,
we may consider a generating element $\alpha$ of $H$.
Since $a$ and $b$ are coprime, there exist integers $h$ and $k$ such that
$ah+bk=1$ by Bezout theorem. It follows that 
$\alpha = a (h \alpha) + b (k \alpha) \in aH+bH$
and therefore $H\subset aH+bH$. The converse inclusion is clear.
\smallskip
 
(ii) Since translating $U$ does not change the cardinalities involved, 
we may freely assume that $0\in U$. Since $\gcd (a,b)=1$, we may write 
$g$ in the form $g=a'b'$ with $\gcd(a,b')=\gcd(b,a')=1$.

We shall consider the decomposition of $U$ as the disjoint union 
of its components in the cosets modulo the subgroup $b'G$ of $G$. 
Let $r$ be the number of cosets $C$ modulo $b'G$ such that intersection 
$U \cap C$ is non-empty.
There exist elements $u_{j}\in U$, sets $X_{j}\subset b'G$ $(j=0,\dots,r-1)$, 
containing $0$ such that $u_{j}-u_{h}\not\in b'G$ if $j\ne h$ and
$$
U=\bigsqcup_{j=0}^{r-1}(u_{j}+X_{j}) = \bigsqcup_{j=0}^{r-1}U_{j},
$$
by writing $U_{j}=u_{j}+X_{j}$.

Let 
$$
V=\{u_{0},\dots,u_{r-1}\}.
$$ 
Since $0\in U$,  we can take $u_{0}=0$.

Let $k$ be a fixed index, $0 \le k \le r-1$. For any $j$, we have 
$$
aU_{j}+bU_{k}=au_{j}+bu_{k}+aX_{j}+bX_{k} \subset au_{j}+bu_{k}+b'G.
$$
It follows that the non-emptiness of $(aU_{j}+bU_{k})\cap(aU_{h}+bU_{k})$ 
implies $au_j+bu_k =au_h+bu_k \pmod{b'}$ and, since $\gcd(a,b')=1$, 
$u_j=u_h$ which finally gives $j=h$. Therefore the sets $aU_{j}+bU_{k}$ 
for $0 \le j \le r-1$ are disjoint. 
Moreover
$$
|aU_{j}+bU_{k}|=|aX_{j}+bX_{k}|\ge|bX_{k}|.
$$
But since $\gcd(g/b',b)=\gcd(a',b)=1$ and $X_{k}\subset b'G$,
we have $|bX_{k}|=|X_{k}|$. From these facts, we deduce
\begin{equation}
\label{equation1}
|aU+bU| = \left|\bigcup_{j,k=0}^{r-1}(aU_{j}+bU_{k})\right| 
\ge \left|\bigsqcup_{j=0}^{r-1}(aU_{j}+bU_{k})\right|=
\sum_{j=0}^{r-1}|aU_{j}+bU_{k}| \ge r|bX_{k}| = r|X_{k}|.
\end{equation}

Since the previous result is valid for any index $k$, 
it follows that
\begin{equation}
\label{equation2}
|aU+bU|\ge r\max_{0\le k\le r-1}|X_{k}|\ge  \sum_{k=0}^{r-1}|X_{k}|=|U|.
\end{equation}
\smallskip

(iii) If the equality $|aU+bU|=|U|$ holds then the inequalities 
in \eqref{equation1} and \eqref{equation2} are equalities. Equality 
in \eqref{equation2} yields $|X_{k}|=|X_{0}|=|U|/r$ for any index $k$.
Equalities in \eqref{equation1} show that for any $k$ we have
\begin{equation}
\label{equation3}
aU+bU=\bigsqcup_{j=0}^{r-1}(au_{j}+bu_{k}+aX_{j}+bX_{k})
=\bigsqcup_{j=0}^{r-1}(au_{j}+bu_{k}+bX_{k})=aV+bu_{k}+bX_{k},
\end{equation}
(here we have used the fact that $0$ belongs to all the $X_j$'s). 
Specializing $k=0$, we get
\begin{equation}
\label{equation4}
aU+bU=\bigsqcup_{j=0}^{r-1}(au_{j}+bX_{0})=aV+bX_{0}.
\end{equation}
We also notice that if we identify the intersection with $b'G$ of 
the second and the third member of \eqref{equation3} (choosing $k=0$) 
we obtain
\begin{equation}
\label{equation5}
aX_0+bX_0=bX_0.
\end{equation}

Both \eqref{equation3} and \eqref{equation4} give decompositions
of $aU+bU$ into unions of subsets of disjoint
cosets modulo $b'G$, hence
for any $j$ and $k$, there exists $h$ such that
\begin{equation}
\label{equation6}
au_{j}+bu_{k}+bX_{k}=au_{h}+bX_{0}.
\end{equation}
Using the facts $X_{k}\subset b'G$ and $\gcd(g,b)\mid b'$  which implies
$\gcd(g/b',b)=1$, we deduce that $X_{k}$ is a translate of $X_{0}$.
Changing if necessary  $u_{k}$, we may now assume that $X_{k}=X_{0}$ for
each index $k$. Letting $X:=X_{0}$, we get 
$$
U=V+X
$$
as announced. The equality $|U|=|V| \times |X|$ follows from $|X_0|=|U|/r$, 
obtained at the very beginning of this proof.

Since $X\subset b' G$ and 
$\gcd( g/b',b)=1$, it is useful to note that $X$ is periodic if and only 
if $bX$ is periodic. But, by assumption, $U$ is not periodic, therefore 
$X$ cannot be periodic either. Hence $bX$ is not periodic, by the previous 
observation. By \eqref{equation5} we have $aX+bX=bX$, hence $aX=\{0\}$ and by Lemma \ref{lemma5.2},
we get  
$$
X\subset \frac{g}{\gcd(g,a)}G.
$$
The non-periodicity of $X$ (which would imply that of $U$) also implies 
with \eqref{equation6} that $au_{j}+bu_{k}=au_{h}$ yielding $aV+bV=aV$.
By Lemma \ref{lemma5.2} again with the fact that $V$ cannot be periodic 
(for the same reason as $X$), we get 
$$
V \subset \frac{g}{\gcd(g,b)}G.
$$
This gives 
$$
U=V+X\subset \gcd \left( \frac{g}{\gcd(g,b)},\frac{g}{\gcd(g,a)} \right) G
=\frac{g}{\gcd(g,a)\times\gcd(g,b)}G.
$$ 
Since $U$ is not included in a proper subgroup of $G$, 
we must have $g=\gcd(g,a)\times\gcd(g,b)$, thus
$g\mid ab$, as asserted.

Conversely, if $U=V+X$ where
$V\subset (g/\gcd(g,b))G$, $X\subset (g/\gcd(g,a))G$,
$g=\gcd(g,a)\times \gcd(g,b)$ and $|U|=|V|\times|X|$,
then clearly $aU+bU=aV+bX$ has cardinality less
than or equal to $|V|\times |X| = |U|$ and 
the equality follows from (ii).
\smallskip

(iv) We let $H=P(U)$ be the period of $U$ in $G$ and denote by $\psi$ 
the canonical homomorphism $G \rightarrow G/H$.  
The assumption implies that $| a U/H + b U/H |= |U/H|$ where 
$U/H=\psi (U)$. We now apply (iii) to the subset $U/H$ 
in the factor group $G/H$ which is isomorphic to $\mathbb{Z}/g_{1}\mathbb{Z}$
where $g_{1}=|G/H|$. We get 
$$
U/H = V_{1}+X_{1}
$$ 
where
$V_{1}\subset a_{1}G/H$, $X_{1}\subset b_{1}G/H$, 
$a_{1}=\gcd(g_{1},a)$, $b_{1}=\gcd(g_{1},b)$ and
$g_{1}=a_{1}b_{1}$
with the property that $|U|=|V_{1}|\times|X_{1}|\times |H|$. 
We infer $|G|=a_{1}b_{1}|H|$ and $H=a_{1}b_{1}G$.
For each coset modulo $H$ in $V_{1}$, we select an
arbitrary representative element in $G$. This gives a subset $V$
of $a_{1}G$ with $|V|=|V_{1}|$. Similarly, we obtain a subset $X$ of $b_{1}G$
formed by representative elements of the cosets modulo $H$ in $X_{1}$.
We conclude that  $U=V+X+H$ with $|U|=|V|\times|X|\times|H|$, as asserted.

Conversely, if $U$ can be written under the form $U=V+X+a_{1}b_{1}G$
for some integers $a_{1}$ and $b_{1}$ dividing respectively $\gcd(g,a)$ and $\gcd(g,b)$
with $V\subset a_{1}G$, $X\subset b_{1}G$, $|U|=|V|\times|X|\times
|a_{1}b_{1}G|$, then the set
$$
aU+bU=aV+bX+a_{1}b_{1}G
$$ 
has cardinality at most equal to
$|V|\times|X|\times|a_{1}b_{1}G|\le |U|$, thus equality occurs by (ii).

 (v) We obviously have $P(U)\subset P(aU+bU)$.
 By the previous point, we have $U=V+X+H$ where $H=P(U)=a_{1}b_{1}G$ is the
 period of $U$ and $V\subset a_{1}G$ and $X\subset b_{1}G$ for two integers
 $a_{1}$ and $b_{1}$ such that
 $a_{1}\mid\gcd(g,a)$ and $b_{1}\mid\gcd(g,b)$. We let
 $U=U_{0}$ and $U_{i+1}=\Gamma_{a,b}(U_{i})$, $i\ge0$.
 The sequence $(P(U_{i}))_{i\ge0}$ is non-decreasing.
 This gives 
 $$
 aU+bU=aV+bX+(aH+bH+bV+aX)=aV+bX+H
 $$
 since $bV,aX\subset H$ and $aH+bH=H$ by (i).
 Let us denote by $\varphi$ the Euler totient function.
 By iterating $k:=\varphi(a)\varphi(b)$ many times this linear operation on $U$
 we get the set 
 $$
 U_{k}=a^{\varphi(a)\varphi(b)}V+b^{\varphi(a)\varphi(b)}X+H.
 $$
 Since $a^{\varphi(b)}\equiv1$ modulo $b$ and $b^{\varphi(a)}\equiv1$ modulo $a$,
 we have $U_{k}=V+X+H=U$. It follows that $P(U)$ contains $P(U_{i})$ for any
 $0\le i\le k$, thus $P(aU+bU)\subset P(U_{0})=H$.
 \end{proof}

\section{\bf Composition and stability for a set of integers with positive upper density}
\label{S5}

The main result of the paper is the following theorem.

\begin{thm}
\label{theorem6.1}
Let $L\ge2$ be an integer, $A$ be an increasing sequence of integers and
assume that  $\overline{d}(A)>0$. 
Let $(a_{j},b_{j})_{j\in\mathbb{N}}$ be a sequence of couples 
of positive integers such that  $a_j\le L$, $b_j\le L$, $\gcd(a_{j},b_{j})=1$ for any $j\ge1$
and $(\Gamma_{a_j,b_j})_{j\in \N}$ be the corresponding  sequence 
of linear operations. We denote
$$
\Gamma_{k}=\bigcirc_{j=1}^k \Gamma_{a_j,b_j},\quad k\in\N,
$$
and $\beta = 1/\overline{d}(A)$.
Let
\begin{equation}
\label{equation7}
K=\lfloor c(\log_{2}(\beta) +L)\rfloor
\end{equation}
be a positive integer and 
$c$ is a sufficiently large absolute constant.
Then

{\rm(i)} there exists a modulus $g$ satisfying
$$
g\le L^{K+1}
$$
such that for any $k\ge K$, $\Gamma_{k}(A)$ is fully periodic modulo $g$,

{\rm(ii)} the sequence $(\Gamma_{k}(A))_{k\ge1}$ is $(K+g^{3}L^2)$-stable.
 \end{thm}

It is good to have in mind Corollary \ref{corollary3.2} when examining
this result.

\begin{proof}
We let $a=a_{K}$, $b=b_{K}$ and $q=\max(a,b)$.
From Lemma \ref{lemma4.1} and since $q\le L$, 
the upper density of $Y:=An$ is at least $1-1/(q+1)$
if we choose $n$ such that
\begin{equation}
\label{equation8}
\log_{2}(n)= \left\lceil \frac{\log_{2}(\beta)}{\log_{2}(3/2)} \right\rceil
+ \left\lceil {\log_{2}(L)} \right\rceil.
\end{equation}

 By Lemma \ref{lemma4.2}, it follows that the gaps in $Z:=aY-bY$ are bounded by $q$,
 thus $\underline{d}Z\ge 1/q$. We thus may apply Kneser's theorem
 (Lemma \ref{lemma4.4}). We infer
 that there exists a positive integer $g_{1}$ such that
 $Z(q+1)$ is semi-periodic modulo $g_{1}$ and 
 $$
 1\ge \underline{d}(Z(q+1))\ge(q+1)\underline{d}Z-\frac q{g_{1}}
 \ge\frac{q+1}{q}-\frac q{g_{1}},
 $$
 hence $g_{1}\le q^2$.

By Lemma \ref{lemma4.3} with $m=n(q+1)$  and in view of \eqref{equation7}
(where $c$ is sufficiently large)
and \eqref{equation8} which imply $K-1\ge2\log_{2}(m)+4L+2$,
there exist two positive integers $\alpha,\beta\le L^{K-1}$ such that
$$
\Gamma_{K-1}(A)=\alpha An(q+1) -\beta An(q+1)+ T=\alpha Y(q+1)-\beta Y(q+1)+T,
$$
where $Y=An$ is the set introduced above and $T$ is a set of integers. 
We have seen that  $Z(q+1)=(aY-bY)(q+1)$ 
is semi-periodic modulo $g_{1}\le q^2$, thus by Lemma \ref{lemma5.1},
$\alpha Z(q+1)-\beta Z(q+1)$ is fully periodic modulo 
$g:=\gcd(\alpha,\beta) g_{1}\le \alpha q^2\le L^{K+1}$. Hence
$$
\Gamma_{K}(A)=a\Gamma_{K-1}(A)-b\Gamma_{K-1}(A)=\alpha Z(q+1)-
\beta Z(q+1)+(aT-bT)
$$ 
is  fully periodic modulo $g$.

We infer that  for any $k\ge K$, the set 
 $\Gamma_{k}(A)=\bigcirc_{j=K}^k \Gamma_{a_j,b_j}(\Gamma_{K-1}(A))$
 is fully periodic modulo $g$. This proves (i).

 Let $U$ be a subset of $G=\mathbb{Z}/g\mathbb{Z}$.
 We first obtain an upper bound for  the number of possible iterates
 of $U$ by some linear operations preserving the cardinality.
 By Proposition \ref{proposition5.3} (iv), a necessary condition for having 
 $|\Gamma_{a,b}(U)|=|U|$ for some
 coprime integers $a$ and $b$ smaller than $L$ is that there exists
 a pair of coprime integers $a'$ and $b'$ dividing $g$ and smaller than
 $L$ such that
 $U$ can be written under the form
 $U=V+X+H$ with $V\subset a'G$ and $X\subset b'G$, $|U|=|V||X||H|$
 and $H=P(U)=a'b'G$.
 Its successive iterates by such linear transformations (i.e. preserving the cardinality)
 $\Gamma_{\lambda a',\mu b'}$
  take the form  $a'' V+b'' X+H$ where $1\le a'',b''\le g$ and
   $\gcd(a'',b')=\gcd(a',b'')=1$.
Hence there are at most $g^2$ such possible iterates of $U$.
Since $a'\le L$ and $b'\le L$,
we deduce that there are at most
$(gL)^{2}$ different iterates of $U$ preserving its cardinality.
It follows that for each integer $k$ between $1$ and $g$,
the number of iterates of $U$ with cardinality $k$ is less than or equal to $(gL)^{2}$,
thus there are  at most $g^{3}L^2$ iterates of $U$.

We denote by $U$ the image of
 ${\Gamma_{K}}(A)$ by the canonical homomorphism of $\mathbb{Z}$
 onto $\mathbb{Z}/g\mathbb{Z}$.
 The discussion above  shows that $U$ has
 at most $g^{3}L^2$ different iterates. Remembering that $\Gamma_{K}(A)$ is
 fully periodic modulo $g$, this gives (ii).
  \end{proof}

\section{\bf Concluding remarks}
\renewcommand{\theenumi}{\arabic{enumi}}
\begin{enumerate}
\item In the case when $(a_{i},b_{i})=(a,b)$ for any $i\ge1$ where
$\gcd(a,b)=1$, we deduce
from Theorem \ref{theorem6.1} (using the same notation) that  
 for any set $A$ of integers 
with positive upper density, there exists an integer $p$ dividing $g$ such that
$\Gamma_{k}(A)=\Gamma_{k+p}(A)$ for any sufficiently large integer $k$.

\item The sequence $\{\Gamma_{k}(A);\ {k\ge1}\}$ needs not
to be eventually periodically stable,
that is periodically stable from some point on (that is $\Gamma_{k+p}(A)=\Gamma_{k}(A)$ for
some $p\ge1$ and any large enough $k$).
Consider for exemple $A=1+3\mathbb{Z}$. Let $\alpha\in(0,1)$ be an irrational
numbers and write $\alpha=0.\alpha_{1}\alpha_{2}\dots$ its dyadic expansion.
We know that the sequence $(\alpha_{i})_{i\ge1}$ is not periodically stable. Put
$(a_{i},b_{i})=(2,1)$ if $\alpha_{i}=0$ and $(a_{i},b_{i})=(3,1)$ otherwise.
Then $\Gamma_{k}(A)=A$ if $\alpha_{k}=0$ and  $-A$ otherwise.
This clearly shows that $\{\Gamma_{k}(A);\ {k\ge1}\}$
is  not eventually periodically stable.

\item For any $\beta >0$, we define $f(\beta)$ to be the
maximum value of $t$ such that there exist a set $A$
and a sequence $(a_{j},b_{j})_{j\ge1}$ with
$\overline{d}(A)>1/\beta$ and $A$ is not $t$-stable
with respect to $\{\Gamma_{a_{j},b_{j}};\,j\ge1\}$. Then
Corollary \ref{corollary3.2} and Theorem \ref{theorem6.1} show that
$\log\log\beta +o(1)< \log(f(\beta))\ll \log \beta$ as
$\beta$ tends to $+\infty$
where the implied constants depend on the bound $L$ for the $a_{j}$'s and the $b_{j}$'s.

\item As for difference set, we can define the restricted
linear transformed set
$\Gamma_{a,b}^+(A)=\Gamma_{a,b}(A)\cap\mathbb{Z}^{+}$ obtained by
considering only the nonnegative elements of the standard linear
transformed set $aA-bA$. A further and more natural question with
respect to Stewart-Tijdeman's and Ruzsa's results \cite{r4,r8} could
be to study the stability of sequences defined by iterating
positive restricted linear operations on a set of integers, but it
is seemingly harder.

\end{enumerate}

\bigskip
\def\refname

\end{document}